\documentclass[12pt]{article}
\usepackage{amsfonts}
\usepackage{amsmath}
\usepackage{mathrsfs}
\usepackage{color}
\usepackage{tikz}
\usepackage{mathrsfs,amscd,amssymb,amsthm,amsmath,bm,graphicx,psfrag,subfigure,url}
\usepackage{authblk}
\usepackage{indentfirst}

\newtheorem{thm}{Theorem}[section]

\newtheorem{lem}[thm]{Lemma}
\newtheorem{cor}[thm]{Corollary}

\begin{document}
\title{\bf The minimum number of maximal dissociation sets in unicyclic graphs}

\author[1]{Junxia Zhang}
\author[2]{Xiangyu Ren}
\author[3,*]{Maoqun Wang}
\affil[1] {\small School of Mathematical Sciences, Xiamen  University, Xiamen 361005, P.R. China} 
\affil[2] {\small School of Mathematical Sciences, Shanxi  University, Taiyuan 030006, P.R. China}
\affil[3] {\small School of Mathematics and Information Sciences, Yantai University, Yantai 264005, P.R. China}
\affil[*] {Corresponding author: \texttt{wangmaoqun@ytu.edu.cn}}
\date{}
\maketitle

\begin{abstract}
A subset of vertices in a graph $G$ is considered a maximal dissociation set if it induces a subgraph with vertex degree at most 1 and it is not contained within any other dissociation sets. In this paper, it is shown that for $n\geq 3$, every unicyclic graph contains a minimum of $\lfloor n/2\rfloor+2$ maximal dissociation sets. We also show the graphs that attain this minimum bound.
\end{abstract}

\vspace{2mm} \noindent{\bf Keywords}: dissociation set; maximal dissociation set; unicyclic graph; extremal graph
\vspace{2mm}

\setcounter{section}{0}
\section{Introduction}\setcounter{equation}{0}

In a graph $G$, a set of vertices is defined as a {\it dissociation set} if it results in a subgraph where the maximum vertex degree is at most 1. This dissociation set is considered {\it maximal} if it is not contained within any other dissociation sets, or {\it maximum} if it has the largest possible cardinality.
Yannakakis \cite{yannakakis} was the first to propose the dissociation set, the scholar extends the familiar independent set and induced matching concepts, see \cite{stock,gener,match2}.
Finding a dissociation set in a specified graph has been proved to be NP-hard, as evidenced in the case of bipartite graphs or planar graphs \cite{np2,yannakakis}.
Furthermore, the concept of dissociation set of graphs can be seen as a case of $k$-improper coloring (also known as defective coloring), see \cite{coven,havet}.

Since 1960s, Erd\H{o}s and Moser gave the upper bound on the number of maximal independent sets for graphs of $n$ vertices.
After that, many researchers have obtained similar results on trees \cite{wilf}, connected graphs excluding $C_3$ \cite{jou2}, graphs with at most one cycle \cite{jou1}, and other graph classes \cite{goh,indepen4}. In 2008, Koh, Goh and Dong \cite{indepen1} determined the upper bound on the number of maximal independent sets for unicyclic connected graphs.
Shortly thereafter, Tu et al. \cite{tu,shi} have studied the upper bound on the number of maximum dissociation sets for trees and graphs excluding $C_3$. Sun and Li \cite{sun} recently obtained the same result for forests with given order and dissociation number, and Zhang et al. \cite{jun} proved the minimum number of maximal dissociation sets of trees.

Based on the research results mentioned above, it is natural to study the lower bound on the number of maximal dissociation sets in unicyclic graphs. Our result is as follows.

The set of all maximal dissociation sets in $G$ and its cardinality are denoted by $MD(G)$ and $\phi (G)$, respectively. Given two non-negative integers $p$ and $q$ with $p\geq q$,
$T(p,q)$ represents a class of trees and the definition can refer to section 1 of \cite{jun}.  Let $U(p,q)$ be the graph obtained by identifying a vertex of $K_3$ and the center of $T(p,q)$. A class of unicyclic graph, denoted by $\mathscr{U}_{r,t}$, obtained from a cycle of order $r$ by adding a pendant vertex to each of its $t$ vertices. Let $U_{r,t}\in \mathscr{U}_{r,t}$, where $r\geq 3$ and $t\geq 0$. It is clear that $U(0,0)=U_{3,0}$ and $U(1,0)=U_{3,1}$.

\begin{thm}\label{th1}
	Let $G$ be a unicyclic graph of order $n$ ($n\geq 3$), we have $\phi(G)\geq \lfloor \frac{n}{2}\rfloor+2$, with equality if and only if one of the following statements holds:\\
\text{(1) }{ $n$ is odd, $G=U\left(\frac{n-3}{2},\frac{n-3}{2}\right)$;}\\
\text{(2) }{ $n$ is even, $G=U\left(\frac{n-2}{2},\frac{n-4}{2}\right)$ (see Figure \ref{fig0}). In addition, when $n=6$ and $n=8$,  $G\in \{U_{6,0},U_{5,1}\}$ and $G=U_{4,4}$
 are also extremal graphs, respectively(see Figure \ref{fig3}).}
\end{thm}

\begin{figure}[h]
  \centering
  \includegraphics[width=110mm]{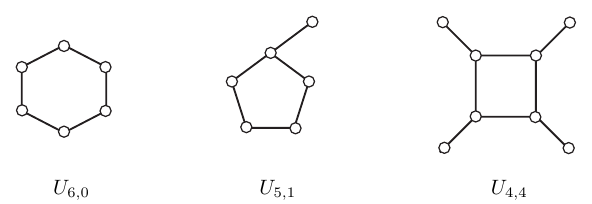}\\
  \caption{$U_{6,0},U_{5,1}$ and $U_{4,4}$.}\label{fig3}
\end{figure}

\begin{figure}[h]
	\centering
	\includegraphics[width=100mm]{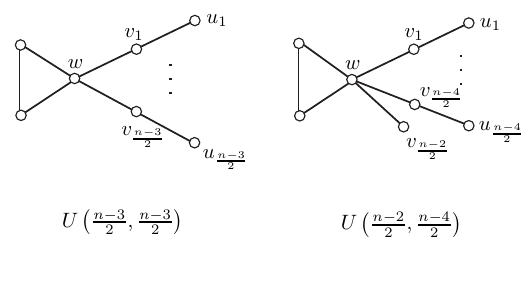}\\
	\caption{The unicyclic graphs $U\left(\frac{n-3}{2},\frac{n-3}{2}\right)$ and $U\left(\frac{n-2}{2},\frac{n-4}{2}\right)$.}\label{fig0}
\end{figure}

\section{Preliminaries}
For a simple graph $G$, the vertex set and edge set are defined as $V=V(G)$ and $E=E(G)$, respectively. A {\it unicyclic} graph contains exactly one cycle. A {\it caterpillar} graph is a tree in which deleting all pendant vertices leaves a path. Let $P_n$ represent the path and $C_n$ represent the cycle. For $v\in V(G)$, we denote by $N_G(v)$ the neighbour set of $v$, i.e., $N_G(v)=\{u\in V|uv\in E\}$. The {\it degree} of $v$ in $G$ is defined as $d_G(v)=|N_G(v)|$, and the closed neighbour set $N_G[v]$ is defined as $N_G[v]=\{v\}\cup N_G(v)$. We will abbreviate $N_G(v)$, $d_G(v)$ and $N_G[v]$ as $N(v)$, $d(v)$ and $N[v]$, respectively, if there is no possibility of confusion.

If a vertex of $G$ with degree 1, it can be called a {\it leaf}. A {\it support vertex} in $G$ is the neighbor of a leaf.
Given a vertex set $U$, we define $G-U$ as the graph obtained by removing the vertices set $U$ from $G$, and $G[U]$ as the subgraph induced by $U$. In this paper, we use $G-u$ to represent $G-\{u\}$. Since $\phi(G)=|MD(G)|$, we have
\begin{align*}
&\phi(G,\overline{v}) \ =|\{S:\ S\in MD(G) \ \textup{and} \ v\notin S\}|,\\
&\phi(G,v)\ =|\{S:\ S\in MD(G) \ \textup{and} \ v\in S\}|,\\
&\phi(G,v^{0})=|\{S:\ S\in MD(G),v\in S \ \textup{and} \ d_{G[S]}(v)=0\}|,\\
&\phi(G,v^{1})=|\{S:\ S\in MD(G),v\in S \ \textup{and} \ d_{G[S]}(v)=1\}|.
\end{align*}

It is obvious that,
$$
\phi (G)=\phi (G,v)+\phi(G,\overline{v})=\phi (G,v^{0})+\phi (G,v^{1})+\phi (G,\overline{v}).
$$

Further, for $\{x_1,\ldots,x_s\}\subseteq V(G)$, say $\widehat{x_i}\in \{\overline{x_i},x_i,x_i^0,x_i^1\}$, we denote
\begin{align*}
\phi(G,\widehat{x_1}\widehat{x_2}\ldots\widehat{x_s})
=|\bigcap_{i=1}^{s} MD(G,\widehat{x_i})|.
\end{align*}

The subsequent lemmas will be frequently utilized in the following discussions.

\begin{lem}\label{lemc}
	Consider the union of two disjoint graphs $G$ and $H$, denoted by $G\cup H$. It follows that $\phi(G\cup H)=\phi(G)\phi(H)$.
\end{lem}

\begin{lem}\label{jun} \cite{jun}
	For $n\geq 3$ and any tree $T$ of order $n$, we have $\phi(T)\geq \lceil \frac{n}{2}\rceil+1$, with equality if and only if $T=T\left(\frac{n}{2},\frac{n-2}{2}\right)$ for even $n$ and $T\in \{T\left(\frac{n-1}{2},\frac{n-1}{2}\right),T\left(\frac{n+1}{2},\frac{n-3}{2}\right)\}$ for odd $n$ (see Figure \ref{fig1}).
\end{lem}
\begin{figure}[h]
	\centering
	\includegraphics[width=121mm]{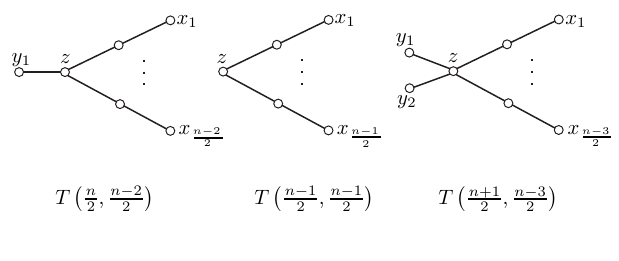}\\
	\caption{The trees $T\left(\frac{n}{2},\frac{n-2}{2}\right)$, $T\left(\frac{n-1}{2},\frac{n-1}{2}\right)$ and $T\left(\frac{n+1}{2},\frac{n-3}{2}\right)$.}\label{fig1}
\end{figure}

By this lemma, the following corollaries can be deduced.

\begin{cor}\label{cor}
For a path $P_n$, $\phi(P_n)\geq \lceil \frac{n}{2}\rceil+1$ with equality if and only if $n=3,4,5$.

\end{cor}

\begin{cor}\label{cor1}
Let $G$ be a caterpillar of order $n$. Then $\phi(G)\geq \lceil \frac{n}{2}\rceil+1$ with equality if and only if $G\in \{T(1,1),T(2,1),T(2,2),T(3,1),T(3,2),T(4,2)\}$.

\end{cor}

\begin{lem}\label{lem0} \cite{jun}
If $u$ is a support vertex of a graph $G$ or $MD(G,u^0)= \emptyset$, then $\phi(G,u^0)=0$.
\end{lem}

\begin{lem}\label{lemG-u,u}
Let $G$ be a graph and $u\in V(G)$. Then $\phi(G-u)\geq \phi(G,\overline{u})$ and $\phi (G-N[u])\geq \phi (G,u^0)$.
\end{lem}
\begin{proof}
Given any maximal dissociation set $S\in MD(G,\overline{u})$, it is also a maximal dissociation set of $G-u$. Thus, $\phi (G-u)\geq \phi (G,\overline{u})$.
Similarly,
given any $S'\in MD(G,u^0)$, $S'\backslash \{u\}\in MD(G-N[u])$. Thus, $\phi (G-N[u])\geq \phi (G,u^0)$.
\end{proof}

\section{Proof of Theorem \ref{th1}}

We are now in a position to prove our main result. We break up the proof into three cases as follows.

{\bf Case 1.} $G\in \mathscr{U}_{r,t}$ and $t=0$.

For a cycle $C_n$, let $v\in V(C_n)$, $C_n-v$ can be denoted by $P_{n-1}$.
When $n=3$, the only unicyclic graph is $C_3$ and $\phi(C_3)=3=\lfloor\frac{3}{2}\rfloor +2$. Since $C_3=U_{3,0}$, the theorem follows for $n=3$. Then we assume that $n\geq 4$ and first present the following lemma.
\begin{lem}
 $\phi(C_n)\geq \phi(P_{n-1})+1$ with equality if and only if $n=6$.
\end{lem}

\begin{proof}
It is straightforward to check that $\phi(C_4)> \phi(P_{3})+1$, $\phi(C_5)> \phi(P_{4})+1$ and $\phi(C_6)=\phi(P_{5})+1$, where $C_6=U_{6,0}$.
Then we show that $\phi(C_n)> \phi(P_{n-1})+1$ for $n>6$.

Given any $S\in MD(P_{n-1})$, $S$ is a dissociation set of $C_n$. If $S\notin MD(C_n)$, then $S'=S\cup \{v\}$ ensures its maximality in $C_n$. This implies that there exists an injection from $S$ to $S'$. On the other hand, for any $C_n$ ($n>6$), there exists at least one $S_1\in MD(C_n,v^0)$ and $S_2\in MD(C_n,v^1)$ that make $S_1\backslash \{v\}$ and $S_2\backslash \{v\}$ not maximal dissociation sets of $P_{n-1}$.
Thus, $$\phi(C_n)\geq \phi(P_{n-1})+2> \phi(P_{n-1})+1$$ for $n>6$, as desired.
\end{proof}

Since $P_{n-1}$ is a special tree with $n-1$ vertices, according to Corallary \ref{cor}, we have
$\phi(P_{n-1})\geq \left\lceil\frac{n-1}{2}\right\rceil +1$ with equality only when $n=4,5,6$.
Thus,
\begin{align*}
\phi(C_n)\geq \phi(P_{n-1})+1\geq \left\lceil\frac{n-1}{2}\right\rceil +2= \left\lfloor\frac{n}{2}\right\rfloor +2.
\end{align*}
All equalities hold if and only if $n=6$.

Above all, the unicyclic graphs $U_{3,0}$ and $U_{6,0}$ achieve the extremal value.

{\bf Case 2.} $G\in \mathscr{U}_{r,t}$ and $t\geq 1$.

Let $y$ be a leaf of $G$, $N[y]=\{x\}$ and $N(x)=\{w,y,z\}$. Denoted $G-N[y]$ by $U$ and $\phi(G-N[y])=\phi(U)$, we show the following lemma.

Since the order of $G$ is $n=r+t\geq 4$, when $n=4$, $G=U_{3,1}$. Further, $\phi(G)=4=\lfloor\frac{4}{2}\rfloor+2$, the theorem holds true for $n=4$. Now we assume $n\geq 5$.
\begin{lem}
$\phi(G)\geq \phi(G-N[y])+2$.
\end{lem}

\begin{proof}
Remember that $\phi(G)=\phi(G,x^0)+\phi(G,x^1)+\phi(G,\overline{x})$. Since $x$ is a support vertex in $G$, $\phi(G,x^0)=0$ by Lemma \ref{lem0}. Then we need to estimate $\phi(G,x^1)$ and $\phi(G,\overline{x})$. Since $\phi(G,x^1)$ satisfies
$$\phi(G,x^1)=\phi(G,x^1y^1)+\phi(G,x^1w^1)+\phi(G,x^1z^1).$$

\begin{figure}[h]
  \centering
  \includegraphics[width=120mm]{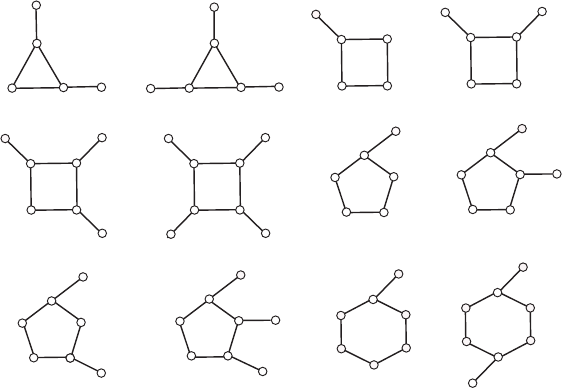}\\
  \caption{}\label{fig2}
\end{figure}

Given any $S\in MD(G,x^1y^1)$, $\{x,y\}\in S$ and $\{w,z\}\cap S=\emptyset$. Then $S\in MD(G,x^1y^1)$ only when $S\backslash \{x,y\}\in MD(U-w-z)$.

Given any $S\in MD(G,x^1w^1)$ and $S\in MD(G,x^1z^1)$, $S\neq \emptyset$ because the order of cycle is at least 3. This implies that $\phi(G,x^1w^1)\geq 1$ and $\phi(G,x^1z^1)\geq 1$.
Hence,
$$\phi(G,x^1)\geq \phi(U-w-z)+2.$$

Given any $S\in MD(G,\overline{x})$, $y\in S$ and $\{w,z\}\cap S\neq \emptyset$. Hence,
$$\phi(G,\overline{x})=\phi(U)-\phi(U,\overline{wz}).$$

Combining Lemma \ref{lemG-u,u}, it holds that
\begin{align*}
\phi(G)=&\phi(G,x^0)+\phi(G,x^1)+\phi(G,\overline{x})\\
       \geq &\phi(U-w-z)+2+\phi(U)-\phi(U,\overline{wz})\\
       =&\phi(U)+\left(\phi(U-w-z)-\phi(U,\overline{wz})\right)+2\\
       \geq &\phi(G-N[y])+2.
\end{align*}
\end{proof}

Note that $G-N[y]$ is a caterpillar of order $n-2$. By Corollary \ref{cor1}, we get
$$\phi(G-N[y])\geq \left\lceil\frac{n-2}{2}\right\rceil +1$$ for $n\geq 5$. Thus,
\begin{align*}
\phi(G)\geq \phi(G-N[y])+2\geq \left\lceil\frac{n-2}{2}\right\rceil +1+2\geq \left\lfloor\frac{n}{2}\right\rfloor +2.
\end{align*}
The second equality holds only when $G-N[y]$ is one of extremal trees in Corollary \ref{cor1} as $G$ is one of the graphs in Figure \ref{fig2}. Since the last equality holds when $n$ is even, we can proceed to calculate the remaining graphs in the Figure \ref{fig2}, and we deserve that all equalities hold if and only if $G\in \{U_{4,4},U_{5,1}\}$.

{\bf Case 3.} $G$ is a unicyclic graph other than Case 1 and Case 2.

\begin{lem}\label{lem2}
Let $U$ be an arbitrary unicyclic graph with $|U|\geq 3$ and $w\in V(U)$ a non support vertex. Let $G_1$ be the unicyclic graph obtained from $U$ by adding $k$ leaves $v_1,\ldots,v_k$ ($k\geq 2$) on $w$. Let $G_2=G_1-wv_k + v_1v_k$, see Figure \ref{fig4}. Then $\phi(G_1)\geq \phi(G_2)$.

\end{lem}

\begin{figure}[h]
  \centering
  \includegraphics[width=120mm]{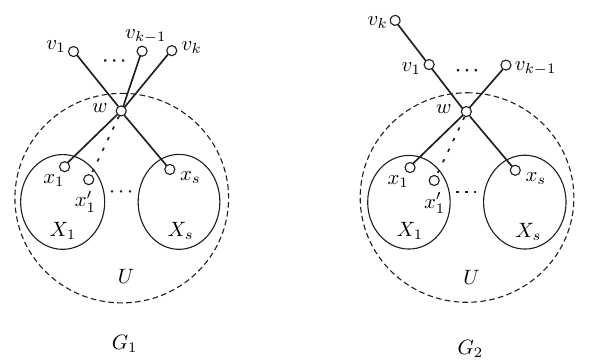}\\
  \caption{The unicyclic graphs $G_1$ and $G_2$.}\label{fig4}
\end{figure}

\begin{proof}
Consider $I=\{x_1,x_2,\ldots,x_s\}$ as the set of neighbors of $w$ in $U$. Let $X_i$ represent the component of $U-w$ that includes $x_i$ for $i\in\{1,2,\ldots,s\}$. If $w$ is not a vertex on the cycle, we assume the cycle is included in $X_1$. While if $w$ is on the cycle, we assume $x_1w$ and $x_1'w$ are two adjacent edges of cycle and consider $I'=\{x_1,x_1',x_2,\ldots,x_s\}$ as the set of neighbors of $w$ in $U$, $x_1$ and $x_1'$ are included in $X_1$ (see Figure \ref{fig4}).
Firstly, we consider the condition that $w$ is not a vertex on the cycle.

Assume that $k> 2$, $w$ is a support vertex in both $G_1$ and $G_2$, we have $\phi (G_1,w^0)=\phi (G_2, w^0)=0$ by Lemma \ref{lem0}. On this condition, we have $\phi(G_2)=\phi(G_2,\overline{w})+\phi(G_2,w^1)$. As the same arguments as before, we estimate $\phi(G_2,\overline{w})$ and $\phi(G_2,w^1)$ as follows.

{\bf Claim 1.} $\phi(G_2,\overline{w})=\phi(G_1,\overline{w})$.

Given any $S_2\in MD(G_2,\overline{w})$, we know that $w\notin S_2$ and $\{v_1,\ldots,v_{k-1},v_k\}\subseteq S_2$. Then $S_2\in MD(G_2,\overline{w})$ only when $S_2\backslash \{v_1,\ldots,v_{k-1},v_k\}\in MD(U-w)$. Thus $$\phi(G_2,\overline{w})=\phi(U-w).$$

The case $S_1\in MD(G_1,\overline{w})$ is similar. We can obtain that $S_1\in MD(G_1,\overline{w})$ only when $S_1\backslash \{v_1,\ldots,v_k\}\in MD(U-w)$.
Thus $$\phi(G_1,\overline{w})=\phi(U-w).$$ Hence, we give $\phi(G_2,\overline{w})=\phi(G_1,\overline{w})$.

{\bf Claim 2.} $\phi(G_2,w^1)=\phi(G_1,w^1)- \phi(U-N[w])$.

Note that $$\phi(G_2,w^1)=\sum_{i\in I}\phi(G_2,w^1x_i^1)+\phi(G_2,w^1v_1^1)+ \sum_{j=2}^{k-1} \phi(G_2,w^1v_j^1).$$

Given any $S_2\in MD(G_2,w^1x_i^1)$, it can be seen that $u_1\in S_2$ and $\{v_1,\ldots,v_{k-1}\}\cap S_2=\emptyset$. Then $S_2\in MD(G_2,w^1x_i^1)$ only when $S_2\backslash \{v_k\}\in MD(U,w^1x_i^1)$.
Hence, $$\sum_{i\in I} \phi(G_2,w^1x_i^1)=\phi(U,w^1).$$

The case $S_2\in MD(G_2,w^1v_1^1)$ and $S_2\in MD(G_2,w^1v_j^1)$ are similar. We can calculate that
$$\phi(G_2,w^1v_1^1)=\phi(U-N[w]).$$
and $$\sum_{j=2}^{k-1} \phi(G_2,w^1v_j^1)=(k-2)\phi(U-N[w]).$$

Thus, $\phi(G_2,w^1)=\phi(U,w^1)+(k-1)\phi(U-N[w])$.

Now we consider $G_1$. It shows that $S_1\in MD(G_1,w^1x_i^1)$ only when $S_1\in MD(U,w^1x_i^1)$ and $S_1\in MD(G_1,w^1v_j^1)$ only when $S_1\backslash \{w,v_j\}\in MD(U-N[w])$ ($j\in \{1,\ldots,k\}$). Hence,
$$\sum_{i\in I} \phi(G_1,w^1x_i^1)=\phi(U,w^1).$$
and
$$\sum_{j=1}^k \phi(G_1,w^1v_j^1)=k \cdot\phi(U-N[w]).$$
Then $$\phi(G_1,w^1)=\phi(U,w^1)+k \cdot\phi(U-N[w]).$$

Hence, $\phi(G_2,w^1)=\phi(G_1,w^1)- \phi(U-N[w])$.

Based on the above claims and $\phi(U-N[w])\geq 1$, we have
\begin{align*}
\phi(G_2)= &\phi(G_2,w^0)+\phi(G_2,w^1)+\phi(G_2,\overline{w})\\
         = &\phi(G_1,w^0)+\phi(G_1,w^1)-\phi(U-N[w])+\phi(G_1,\overline{w})\\
         = &\phi(G_1)-\phi(U-N[w])\\
         < &\phi(G_1).
\end{align*}

When $k=2$, $w$ is a support vertex in $G_1$ but not in $G_2$. If $MD(U,w^0)= \emptyset$, we have $\phi (G_2,w^0)=\phi (G_1,w^0)=0$.
Using the same argument as before, for the remaining items, we obtains
\begin{align*}
\phi(G_2,\overline{w})=&\phi(G_1,\overline{w}),\\
\phi(G_2,w^1)=&\phi(G_1,w^1)- \phi(U-N[w]).
\end{align*}

We can directly obtain that
\begin{align*}
\phi(G_2)= &\phi(G_2,w^1)+\phi(G_2,w^0)+\phi(G_2,\overline{w})\\
         = &\phi(G_1,w^1)- \phi(U-N[w])+\phi(G_1,w^0)+\phi(G_1,\overline{w})\\
         = &\phi(G_1)-\phi(U-N[w])\\
         < &\phi(G_1).
\end{align*}

Then we consider $MD(U,w^0)\neq \emptyset$. Also, $\phi (G_1,w^0)=0$ by Lemma \ref{lem0}.
Given any $S_2\in MD(G_2,w^0)$, $v_1\notin S_2$ and $v_k\in S_2$. Then $S_2\in MD(G_2,w^0)$ only when
$S_2\backslash \{v_k\}\in MD(U,w^0)$. It is clear that $$\phi (G_2,w^0)=\phi(U,w^0).$$ Hence,
$\phi (G_2,w^0)=\phi (G_1,w^0)+\phi(U,w^0)$.

We have the following by Lemma \ref{lemG-u,u}:
\begin{align*}
\phi(G_2)= &\phi(G_2,w^1)+\phi(G_2,w^0)+\phi(G_2,\overline{w})\\
         = &\phi(G_1,w^1)- \phi(U-N[w])+\phi(G_1,w^0)+\phi(U,w^0)+\phi(G_1,\overline{w})\\
         = &\phi(G_1)-\left( \phi(U-N[w])-\phi(U,w^0)\right)\\
         \leq &\phi(G_1),
\end{align*}
with equality holds only when $\phi(U-N[w])=\phi(U,w^0)$.

The above discussion illustrates that $\phi(G_1)\geq \phi(G_2)$ when $w$ is not on the cycle. In fact, when $w$ is on the cycle, we just need to replace the neighbor set $I$ of $w$ with $I'$ and keep others unchanged. The same arguments demonstrate that the above claims hold regardless of whether $w$ is the vertex on the cycle, and the lemma has been proved.
\end{proof}

Let $G$ be a unicyclic graph, which attains the lower bound on the number of all maximal dissociation sets. If $G$ has no leaf, then $G=C_3$. If $G$ has a leaf $x$ and $y$ is the support vertex, it is clear $d_G(y)\geq 2$. Consider a vertex $z$ adjacent to $y$, distinct from $x$. Now we just need to consider that $G$ is a unicyclic graph not in case 1 and case 2.

If $d_G(y)>2$ and $|G-N[y]\backslash \{z\}|=1$, then $y$ must be a vertex of a triangle. We deduce that $\phi(G)> \phi(G-yz+xz)$, contradicting the definition of $G$.
If $|G-N[y]\backslash \{z\}|=2$, then $y$ must on the cycle. When $G$ has no pendent edges, we have $\phi(G)\geq \phi(G-yz+xz)$. If the equality holds,
we can shift $G$ to a unicyclic graph with a pendent edge. If not, a contradiction arises regarding the definition of $G$.

Now we consider
$|G-N[y]\backslash \{z\}|>2$ and $G$ has no support vertex with degree 2. Treat $G$ as $G_1$, we have $\phi(G)\geq \phi(G_2)$ by Lemma \ref{lem2}. The same argument as before, when $\phi(G)= \phi(G_2)$, we can shift $G$ to a unicyclic graph with a pendent edge. When $\phi(G)> \phi(G_2)$, we have a contradiction. That implies that $G$ has a support vertex with degree 2, and has the form showed in Figure \ref{fig5}.

\begin{figure}[h]
  \centering
  \includegraphics[width=90mm]{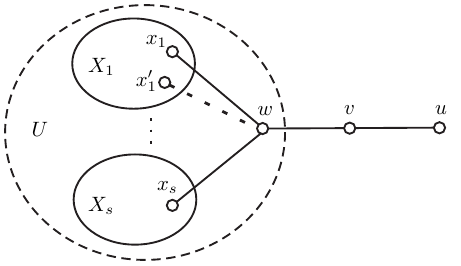}\\
  \caption{The extremal unicyclic graph $G$.}\label{fig5}
\end{figure}

\begin{lem}\label{G-uv}
	Let $G$ be a unicyclic graph showed in Figure \ref{fig5}. Then $\phi(G)\geq \phi(G-\{u,v\})+1$.
\end{lem}

\begin{proof}
We define $I=\{x_1,x_2,\ldots,x_s\}$, $I'=\{x_1,x_1',\ldots,x_s\}$ and $X_1,\ldots,X_s$ as in the proof of Lemma \ref{lem2}. Let $U=G-\{u,v\}$.

Since $\phi(G)=\phi(G,w^0)+\phi(G,w^1)+\phi(G,\overline{w})$, we now consider each term on the right-hand side of the equality.

{\bf Claim 1.} $\phi(G,w^0)=\phi(G-\{u,v\},w^0)$.

We observe that $S\in MD(G,w^0)$ only when $S\backslash \{u\}\in MD(U,w^0)$ regardless of whether $w$ is on the cycle or not. The claim can be readily inferred.

{\bf Claim 2.} $\phi(G,w^1)=\phi(G-\{u,v\},w^1)+1$.

Note that $$\phi(G,w^1)=\sum_{i\in I}\phi(G,w^1x_i^1)+\phi(G,w^1v^1)$$ or $$\phi(G,w^1)=\sum_{i\in I'}\phi(G,w^1x_i^1)+\phi(G,w^1v^1).$$

Given any $S\in MD(G,w^1x_i^1)$ ($i\in I$ or $i\in I'$), there must exist $u\in S$ and $v\notin S$. We know that $S\in MD(G,w^1x_i^1)$ only when $S\backslash \{u\}\in MD(U,w^1)$. Hence,
$$\sum_{i\in I}\phi(G,w^1x_i^1)=\phi(U,w^1)$$ or $$\sum_{i\in I'}\phi(G,w^1x_i^1)=\phi(U,w^1).$$

Given any $S\in MD(G,w^1v^1)$, there is $\left(N[w]\cup N[v]\right)\cap S=\{w,v\}$. We know that $S\in MD(G, w^1v^1)$ only when
$S\backslash \{w,v\}\in MD(U-N[w])$. Furthermore, $\phi(G,w^1v^1)\geq 1$.

Thus we have $\phi(G,w^1)\geq \phi(U,w^1)+1=\phi(G-\{u,v\},w^1)+1$.

{\bf Claim 3.} $\phi(G,\overline{w})\geq \phi(G-\{u,v\},\overline{w})$.

Given any $S'\in MD(G-\{u,v\},\overline{w})$, $S'\notin MD(G)$ and let $S=S'\cup \{u,v\}$. Consequently, $S$ must be maximal in $G$ if $S'$
is maximal in $G-\{u,v\}$, indicating the existence of an injection from $S'$ to $S$. Then we have $\phi(G,\overline{w})\geq \phi(G-\{u,v\},\overline{w})$.

Based on the above discussions, we get
\begin{align*}
\phi(G)=&\phi(G,w^0)+\phi(G,w^1)+\phi(G,\overline{w})\\
       \geq &\phi(G-\{u,v\},w^0)+\phi(G-\{u,v\},w^1)+1+\phi(G-\{u,v\},\overline{w})\\
       =&\phi(G-\{u,v\})+1.
\end{align*}

The proof of the lemma is finished.
\end{proof}

Now let us focus on the proof of Case 3 again. Based on the induction hypothesis,
$$\phi(G-\{u,v\})\geq \left\lfloor\frac{n-2}{2}\right\rfloor+2=\left\lfloor\frac{n}{2}\right\rfloor+1$$
and all the equality holds only when $G-\{u,v\}\in \{U_{6,0},U_{5,1},U_{4,4}\}$ or $G-\{u,v\}=U(\frac{n-5}{2},\frac{n-5}{2})$ for odd $n$ and $G-\{u,v\}=U(\frac{n-4}{2},\frac{n-6}{2})$ for even $n$.

If $G-\{u,v\}\in \{U_{6,0},U_{5,1},U_{4,4}\}$, for any $w\in G-\{u,v\}$, we have calculated directly that $\phi(G)>\left\lfloor\frac{n}{2}\right\rfloor+2$. Under these circumstances where $G-\{u,v\}$ is either $U(\frac{n-5}{2},\frac{n-5}{2})$ or $U(\frac{n-4}{2},\frac{n-6}{2})$, we will demonstrate that $\phi(G)$ achieves the lower bound $\left\lfloor\frac{n}{2}\right\rfloor+2$ only when $w$ functions as the center vertex in $G-\{u,v\}$.

{\bf Subcase 1.} $w$ is the leaf of $G-\{u,v\}$.

On this condition, we can calculate directly that
 $\phi(G)=\frac{3n-1}{2}>\left\lfloor\frac{n}{2}\right\rfloor+2$ if $G-\{u,v\}=U(\frac{n-5}{2},\frac{n-5}{2})$, and
$\phi(G)=\frac{3n+2}{2}>\left\lfloor\frac{n}{2}\right\rfloor+2$ or $\phi(G)=\frac{n+6}{2}>\left\lfloor\frac{n}{2}\right\rfloor+2$ if $G-\{u,v\}=U(\frac{n-4}{2},\frac{n-6}{2})$.

{\bf Subcase 2.} $w$ is on the triangle excluding the center of $G-\{u,v\}$.

On this condition, we can calculate directly that $\phi(G)=\frac{n+5}{2}>\left\lfloor\frac{n}{2}\right\rfloor+2$ if $G-\{u,v\}=U(\frac{n-5}{2},\frac{n-5}{2})$,
and $\phi(G)=\frac{n+6}{2}>\left\lfloor\frac{n}{2}\right\rfloor+2$ if $G-\{u,v\}=U(\frac{n-4}{2},\frac{n-6}{2})$.

{\bf Subcase 3.} $w$ is the center of $G-\{u,v\}$.

On this condition, we can calculate directly that $\phi(G)=\left\lfloor\frac{n}{2}\right\rfloor+2$ if $G-\{u,v\}=U(\frac{n-5}{2},\frac{n-5}{2})$
or $G-\{u,v\}=U(\frac{n-4}{2},\frac{n-6}{2})$.

{\bf Subcase 4.} $w$ is neither the leaf nor on the triangle of $G-\{u,v\}$.

On this condition, we can calculate directly that $\phi(G)=\frac{n+5}{2}>\left\lfloor\frac{n}{2}\right\rfloor+2$ if $G-\{u,v\}=U(\frac{n-5}{2},\frac{n-5}{2})$,
and $\phi(G)=\frac{n+6}{2}>\left\lfloor\frac{n}{2}\right\rfloor+2$ if $G-\{u,v\}=U(\frac{n-4}{2},\frac{n-6}{2})$.

By combining Case 1 and Case 2, we complete the proof of Theorem \ref{th1}.

\section{Date Availability}

No date was used for the research described in the article.

\section{Statements and Declarations}

This work was supported by the National Natural Science Foundation of China [Grant number, 12301455], Natural Science Foundation of Shanxi Province
[Grant number, 202203021212484]  and Natural Science Foundation of Shandong Province [Grant number, ZR2023QA080]. The authors declare that they have no conflict of interest.

\end{document}